\documentclass{amsart}
\usepackage{graphicx}
\usepackage{color}
\usepackage[colorlinks]{hyperref}
\usepackage{amsmath,amssymb}
\newtheorem{theorem}{Theorem}[section]
\newtheorem{lemma}[theorem]{Lemma}

\theoremstyle{definition}
\newtheorem{definition}[theorem]{Definition}

\theoremstyle{remark}
\newtheorem{remark}[theorem]{Remark}

\numberwithin{equation}{section}

\begin{document}

\title[Galerkin method  \ldots]{Galerkin method for linear Integral-Algebraic
Equations of index $1$}

\author{B. Shiri}
\address{Faculty of mathematical science, University of Tabriz,
Tabriz - Iran}
\curraddr{Faculty of mathematical science, University of Tabriz,
Tabriz - Iran}
\email{shiri@tabrizu.ac.ir}
\begin{abstract}
In this paper, we study direct and indirect  Galerkin method for solving linear Integral-Algebraic
Equations of index $1.$  Convergence of indirect method is also analyzed.
\end{abstract}
\maketitle
\section{Introduction}
 Recently numerical solution for Integral-Algebraic Equation (IAE)
has been studied by many researchers \cite{1}. In physical models with constrained time or  space variables,
 Algebraic Equations including Integral  or Differential
 Equations arise. Therefore, the investigation of mixed types of equations is important. This can
be seen by observing the number of papers that are about Differential-Algebraic
Equations, Integro Differential Equations, and so on. One of these systems which
mixes Volterra Integral Equations of first and second type is the system of Integral-
Algebraic Equations. However, we don't find any application for this systems except
\cite{6,7,8,9}, but it seems that they be very important in the future. In this paper we
use the standard Galerkin method, for the system of Integral Algebraic Equations
of index 1. We implement the Galerkin method by considering the structure of IAE
of index 1 in two ways, direct and indirect and we analyze indirect method here.
We Consider the system
\begin{equation}\label{s11}
A(t)y(t)+\int_0^tk(t,s)y(s)ds=f(t), \ \ \ \ t\in I:=[0,T],
\end{equation}
where $A\in \mathbf{C}(I,\mathbb{R}^{r\times r}),$ $f\in \mathbf{C}(I,\mathbb{R}^{r})$ and
$k\in \mathbf{C}(\mathbb{D},\mathbb{R}^{r\times r})$ with $\mathbb{D}:=\{(t,s):0\leq s\leq t\leq T\}.$
If $A(t)$ is nonsingular for all $t\in I$ then multiplying (\ref{s11}) by $A.^{-1}$ changes
it to a system of the second kind Volterra Integral Equations that it's theoretical
and numerical analysis have been almost completely investigated \cite{3,4}. Otherwise,
if $A(t)$ is a singular matrix with constant rank for all $t\in I$, then, the system (\ref{s11})  is
an IAE. The classification of this system done by introducing index notation. The
differential index of differential-algebraic equation can be extended to this system.
The following extension is done by Gear in \cite{5}.
\begin{definition} We say the system (\ref{s11}) has index m if, m is the minimum possible
number of differentiating (\ref{s11}) to obtain a system of second kind Volterra Integral
equation.
\end{definition}

For example the following system
\begin{equation}\label{e1a}
x(t)=f_1(t)+\int_0^tk^{11}(t,s)x(s)+k^{12}(t,s)y(s)ds,
\end{equation}
\begin{equation}\label{e1b}
0=f_2(t)+\int_0^tk^{21}(t,s)x(s)+k^{22}(t,s)y(s)ds.
\end{equation}
with the functions $f_i$, $k^{ij}$, \ \ $i,j=1,2$ are sufficiently smooth
and $f_2(0)=0$, $|k^{22}(t,t)|\geq k_0>0$ for all $t\in [0,T];$ (\ref{e1a}) and (\ref{e1b}) is the system of Integral-Algebraic Equations of index 1. For this system the Piecewise Polynomial Collocation
method is investigated by Kauthen in \cite{2}. In this paper we investigate the
Galerkin method for this system.
\section{preliminary}
The Galerkin method is a projection type method \cite{3, 4} that we apply to IAE
of index 1. Let
 $\mathcal{K}:\mathbf{X}\rightarrow\mathbf{X}$ be an operator, where
 $\mathbf{X}$  is a  Hilbert space with orthonormal base
 $\{\phi_i\}_{i=1}^{\infty}$  and $\mathbf{X}_n$ be an $n$-dimensional subspace of X
 generated by $\{\phi_i\}_{i=1}^{n}.$ Then the system (\ref{e1a}) and (\ref{e1b}) changes to
 \begin{equation}\label{e1}
\mathcal{K}x=f.
\end{equation}
For using the Galerkin method's we define the projection $P_n:\mathbf{X}\rightarrow\mathbf{X}_n$  by $P_nx=\sum_{i=1}^n\langle x,\phi_i\rangle \phi_i.$ Then it is clear that $\|P_nx-x\|\rightarrow0$ as $n\rightarrow\infty$. We
search $\mathbf{x}_n\in \mathbf{X}_n$,
$\mathbf{x}_n=\overline{x}\overline{\phi}$ with
$\overline{x}=(x_1,\ldots,x_n)$ and
$\overline{\phi}=(\phi_1,\ldots,\phi_n)^T,$ such that
\begin{equation}\label{e2} P_n\mathcal{K}\mathbf{x}_n=P_nf. \end{equation}
The corresponding operator of the system (\ref{s11}) is defined by
\begin{equation}\label{e6}
\mathcal{K}([x,y])=\left (x-\int_0^tk^{11}(t,s)x(s)+k^{12}(t,s)y(s)ds, \
-\int_0^tk^{21}(t,s)x(s)+k^{22}(t,s)y(s)ds\right)^T,
\end{equation}
and the assumption that $f=(f_1,f_2)^T$, $\mathbf{X}\subseteq(\mathcal{L}^2[0,T])$ and
the functions $f_i$, $k^{ij}$, \ \ $i,j=1,2$ are sufficiently smooth
and $f_2(0)=0$, $|k^{22}(t,t)|\geq k_0>0$ for all $t\in [0,T],$  guarantee us that
 $\mathcal{K}x=f$ has a continuous solution as it follows. We
first differentiate equation (\ref{e1b}) w. r. to t and we have
\begin{equation}
\begin{split} \notag
0=&f^{'}_2(t)+k^{21}(t,t)x(t)+k^{22}(t,t)y(t)\\
&+\int_0^t \frac{\partial k^{21}(t,s)}{\partial t}x(s)+\frac{\partial k^{22}(t,s)}{\partial t} y(s)ds.
\end{split}
\end{equation}
Because $|k^{22}(t,t)|\geq k_0>0$ we obtain
\begin{equation}
\begin{split}\label{e1c}
y(t)=&-\frac{f^{'}_2(t)}{k^{22}(t,t)}-\frac{k^{21}(t,t)}{k^{22}(t,t)}x(t)\\
&-\int_0^t\frac{1}{k^{22}(t,t)}
\frac{\partial k^{21}(t,s)}{\partial
t}x(s)-\frac{1}{k^{22}(t,t)}\frac{\partial k^{22}(t,s)}{\partial t}
y(s)ds,
\end{split}
\end{equation}
as a Volterra integral equations of the second kind for $y(t)$. Now it is easy to show that
this system has a unique solution (see \cite{4}).
\section{Implementation and Convergence}
We can implement the Galerkin method for two systems
(\ref{e1a}),(\ref{e1b}) and (\ref{e1a}),(\ref{e1c}). For the former the Galerkin method  is called direct and for the further, is called indirect method.  For the second system, the convergence and order of convergence can be obtained by similar arguments as discussed \cite{6} and \cite{3}. In \cite{3} the order of
convergence  has been explained for Fredholm
integral equation. Since,  by using  suitable kernel,  each
Volterra integral equation can be changed to a Fredholm integral equation,
 Galerkin method for both are the same. Hence the convergence order of the
approximate solution is equal to the convergence order
of the best approximation in space $X_n$. This is confirmed by numerical examples. Let $\{V_i\}_{i=1}^\infty$ be Legendre polynomials as
orthonormal base for $\mathcal{L}^2[-1,1]$. For $\mathcal{L}^2[0,T]$
we use shifted Legendre polynomials $\{V_{T_i}\}_{i=1}^{\infty}$,
$V_{T_i}(s)=\sqrt{\frac{2}{T}}V_i(\frac{2*s-T}{T})$, $s\in [0,T]$ and
for $(\mathcal{L}^2[0,T])^2,$ we use
$\{(V_{T_i},0),(0,V_{T_i})\}_{i=1}^\infty$ with
$\langle(x,y),(z,t)\rangle=\langle x,z\rangle_2+\langle y,t\rangle_2$ and
$\langle x,y\rangle_2=\int_0^Tx\overline{y}dt$. Also we suppose
$P_n([x,y])=\left(\sum_{i=1}^nx_iV_{T_i},\sum_{i=1}^ny_iV_{T_i}\right).$
For equations (\ref{e1a}), (\ref{e1b})  and (\ref{e1c})  we have.
\begin{equation}\label{e2a}
\begin{split}
\sum_{i=1}^nx_iV_{T_i}(t)=&f_1(t)+\sum_{i=1}^nx_i\int_0^tk^{11}(t,s)V_{T_i}(s)ds\\
&+\sum_{i=1}^ny_i\int_0^tk^{12}(t,s)V_{T_i}(s)ds,
\end{split}
\end{equation}
\begin{equation}\label{e2b}
0=f_2(t)+\sum_{i=1}^nx_i\int_0^tk^{21}(t,s)V_{T_i}(s)ds+\sum_{i=1}^ny_i\int_0^tk^{22}(t,s)V_{T_i}(s)ds,
\end{equation}
\begin{equation}\label{e2c}
\begin{split}
\sum_{i=1}^ny_iV_{T_i}(t)=&-\frac{f^{'}_2(t)}{k^{22}(t,t)}-\frac{k^{21}(t,t)}{k^{22}(t,t)}\sum_{i=1}^nx_iV_{T_i}(t)\\
&-\sum_{i=1}^nx_i \int_0^t\frac{1}{k^{22}(t,t)} \frac{\partial
k^{21}(t,s)}{\partial t}V_{T_i}(s)ds\\
&-\sum_{i=1}^ny_i
\int_0^t\frac{1}{k^{22}(t,t)}\frac{\partial k^{22}(t,s)}{\partial
t}V_{T_i}(s)ds.
\end{split}
\end{equation}
Now we can use (\ref{e2a}),(\ref{e2b}) for the first system and
(\ref{e2a}),(\ref{e2c}) for the second system introduced above. It is
important for analysis to note that  formula (\ref{e2c}) is
derivative of (\ref{e2b}). Now if we multiply equations
(\ref{e2a})-(\ref{e2c}) to $V_{T_j}(t)$ and integrate on the $[0, T]$, we
have obtained
\begin{equation}\label{e3a}
\begin{split}
x_j=&\int_0^Tf_1(t)V_{T_j}(t)dt+\sum_{i=1}^nx_i\int_0^T\int_0^tk^{11}(t,s)V_{T_i}(s)V_{T_j}(t)dsdt\\
&+\sum_{i=1}^ny_i\int_0^T\int_0^tk^{12}(t,s)V_{T_i}(s)V_{T_j}(t)dsdt,
\end{split}
\end{equation}
\begin{equation}\label{e3b}
\begin{split}
0=&\int_0^Tf_2(t)V_{T_j}(t)dt+\sum_{i=1}^nx_i\int_0^T\int_0^tk^{21}(t,s)V_{T_i}(s)V_{T_j}(t)dsdt \\
&+\sum_{i=1}^ny_i\int_0^T\int_0^tk^{22}(t,s)V_{T_i}(s)V_{T_j}(t)dsdt,
\end{split}
\end{equation}
\begin{equation}\label{e3c}
\begin{split}
y_j=&-\int_0^T\frac{f^{'}_2(t)}{k^{22}(t,t)}dt-\sum_{i=1}^nx_i\int_0^T\int_0^t\frac{1}{k^{22}(t,t)}\frac{\partial k^{21}(t,s)}{\partial t}V_{T_i}(s)V_{T_j}(t)dsdt\\
&-\sum_{i=1}^nx_i\int_0^T\frac{k^{21}(t,t)}{k^{22}(t,t)}V_{T_i}(t)V_{T_j}(t)dt\\
&-\sum_{i=1}^ny_i
\int_0^T\int_0^t\frac{1}{k^{22}(t,t)}\frac{\partial
k^{22}(t,s)}{\partial t}V_{T_i}(s)V_{T_j}(t)dsdt.
\end{split}
\end{equation}
The direct and indirect  Galerkin method  for system (\ref{e1a}),(\ref{e1b}) are solving  the linear systems (\ref{e3a}),(\ref{e3b}) and
(\ref{e3a}),(\ref{e3c}) respectively.
\begin{lemma}\label{s21}
 \cite{4} Let $V$ and $W$ be normed spaces, with $W$ complete. Let
$K\in\mathcal{L}(V,W)$, let $\{K_n\}$ be a sequence of compact operators in $\mathcal{L}(V,W),$
and assume $K_n\rightarrow K$ in $\mathcal{L}(V,W).$ Then K is compact.
  \end{lemma}
   \begin{lemma}\label{s22} \cite{4} Let $V$ be a Banach space, and let $\{P_n\}$ be a family of
bounded projections on $V$ with
$$P_nu \rightarrow u \ \ \ \mbox{as} \ \ \ n\rightarrow \infty , u\in V$$
If $K:V\rightarrow V$ is compact, then
$$\|K-P_nK\|\rightarrow 0 \ \ \ \mbox{as} \ \ \ \ n\rightarrow \infty .$$
\end{lemma}
\begin{theorem}\label{s23} \cite{4}
Assume $K:V\rightarrow V$ is bounded, with V a Banach space;
and assume $\lambda-K:V\rightarrow V$ is one to one and  surjective.
Further assume
$$\|K-P_nK\|\rightarrow 0 \ \mbox{as} \  n \rightarrow \infty.$$
Then for all sufficiently large $n$, say, $n\geq N,$ the operator $(\lambda-K)^{-1}$
exists as a bounded operator from $V$ to $V.$ Moreover, it is uniformly
bounded:
$$\sup_{n\geq N} \|(\lambda-P_nK)^{-1}\|\leq \infty.$$
For the solutions $u_n$ (n sufficiently large)  of $P_n(\lambda-K)u_n=P_nF$ and the solutions $u$  of $(\lambda-K)u=F,$
we have
$$u-u_n=\lambda(\lambda-P_nK)^{-1}(u-P_nu)$$
and the two-sided error estimate
$$\frac{|\lambda|}{\|\lambda-P_nK\|}\|u-P_nu\|\leq \|u-u_n\|\leq |\lambda|\left\|(\lambda-P_nK)^{-1}\right \| \|u-P_nu\|.$$
\end{theorem}
We have following Theorem for indirect method.
\begin{theorem}
Let for the system (\ref{e1a}) and (\ref{e1b}) the following conditions are satisfied:
\begin{enumerate}
          \item $f_i\in C^1(0,T), \ K^{i,j}[t,s]\in C^1(\mathbb{D}),$ for all $i,j=\{1,2\},$
          \item $k^{22}(t,t)\geq k_0>0,$
        \end{enumerate}
   where $\mathbb{D}:=\{(t,s): 0\leq s\leq t\leq T\}.$ Then  the convergence order of the
approximate solution is equal to the convergence order
of the best approximation in the space $X_n=\{V_{T_i}\}_{i=1}^\infty.$
\end{theorem}
\begin{proof}
 The space $\mathcal{L}^2[0,T]$ with $\mathcal{L}^2$ Norm is Banach space, hence the Cartesian product on this space, i.e., $V=(\mathcal{L}^2[0,T])^2,$ with introduced norm generate also Banach space.
Now we change  the system  (\ref{e1a}) and (\ref{e1c})  to the following  system
\begin{equation}\label{INF}
(I-K)x=F, \ \ F=[F_1,F_2]^T, \ \ x=[x_1, x_2]^T,
\end{equation}
with
$$F_1(t)=f_1(t), \ \ \ \ \ \ F_2(t)=-\frac{f^{'}_2(t)}{k^{22}(t,t)}-\frac{k^{21}(t,t)}{k^{22}(t,t)}f_1(t),$$
$$K^{11}(t,s)=\left \{ \begin{array}{lc}
                    k^{11}(t,s), & s\leq t, \\
                    0, & s> t,
                  \end{array} \right. \ \  K^{12}(t,s)=\left\{ \begin{array}{lc}
                    k^{12}(t,s), & s\leq t, \\
                    0, & s> t,
                  \end{array} \right. \ \ $$
$$K^{21}(t,s)=\left \{ \begin{array}{lc}
                    -\frac{k^{21}(t,t)}{k^{22}(t,t)}k^{11}(t,s)-\frac{1}{k^{22}(t,t)}\frac{\partial k^{21}(t,s)}{\partial t}, & s\leq t, \\
                    0, & s> t,
                  \end{array}
\right.$$
$$ K^{22}(t,s)=\left\{\begin{array}{lc}
                -\frac{k^{21}(t,t)}{k^{22}(t,t)}k^{12}(t,s)-\frac{1}{k^{22}(t,t)}\frac{\partial k^{22}(t,s)}{\partial t}, & s\leq t, \\
                0. & s> t,
              \end{array}\right.
$$
and with the integral operator
$$Kx=\int_0^T\left[
      \begin{array}{cc}
        K^{11}(t,s) & K^{12}(t,s) \\
        K^{21}(t,s) & K^{22}(t,s) \\
      \end{array}
    \right]\left[
             \begin{array}{c}
               x_1(s) \\
               x_2(s) \\
             \end{array}
           \right]ds.
$$
By introducing
  $$K_nx:=\int_0^T\left[
      \begin{array}{cc}
        K^{11}_n(t,s) & K^{12}_n(t,s) \\
        K^{21}_n(t,s) & K^{22}_n(t,s) \\
      \end{array}
    \right]\left[
             \begin{array}{c}
               x_1(s) \\
               x_2(s) \\
             \end{array}
           \right]ds,
$$ with continuous kernels
$$K^{11}_n(t,s)=\left \{ \begin{array}{lc}
                    k^{11}(t,s), & s\leq t, \\
                    -n(k^{11}(t,t))((t+\frac{1}{n})-s), & s\in [t, t+\frac{1}{n}],\\
                0, & s> t+\frac{1}{n},
                  \end{array} \right. $$
                  $$  K^{12}_n(t,s)=\left\{ \begin{array}{lc}
                    k^{12}_n(t,s), & s\leq t \\
                     -n(k^{12}(t,t))((t+\frac{1}{n})-s), & s\in [t, t+\frac{1}{n}],\\
                0, & s> t+\frac{1}{n},
                  \end{array} \right. \ \ $$
$$K^{21}_n(t,s)=\left \{ \begin{array}{lc}
                    -\frac{k^{21}(t,t)}{k^{22}(t,t)}k^{11}(t,s)-\frac{1}{k^{22}(t,t)}\frac{\partial k^{21}(t,s)}{\partial t}, & s\leq t, \\
                      -n(\frac{k^{21}(t,t)}{k^{22}(t,t)}k^{11}(t,t)+\frac{1}{k^{22}(t,t)}\frac{\partial k^{21}(t,t)}{\partial t})((t+\frac{1}{n})-s), & s\in [t, t+\frac{1}{n}],\\
                    0, & s> t+\frac{1}{n},
                  \end{array}
\right.$$
$$ K^{22}_n(t,s)=\left\{\begin{array}{lc}
                -\frac{k^{21}(t,t)}{k^{22}(t,t)}k^{12}(t,s)-\frac{1}{k^{22}(t,t)}\frac{\partial k^{22}(t,s)}{\partial t}, & s\leq t, \\
                -n(\frac{k^{21}(t,t)}{k^{22}(t,t)}k^{12}(t,t)+\frac{1}{k^{22}(t,t)}\frac{\partial k^{22}(t,t)}{\partial t})((t+\frac{1}{n})-s), & s\in [t, t+\frac{1}{n}],\\
                0, &  s> t+\frac{1}{n}.
              \end{array}\right.
$$
which is compact  and using  lemma \ref{s21}, imply  K is compact.

   Applying the Galerkin method to the Fredholm  system (\ref{INF}) leads to the linear system equivalent with the system (\ref{e3a}),(\ref{e3c}). So it is enough to show that for this system the convergence order of the
approximate solution by the Galerkin method is equal to the convergence order
of the best approximation.  Since, the system (\ref{INF}) also is a volterra system, using the same line of \cite{4}, $(I-K)$ is invertible so is surjective. The Lemma \ref{s22}
 provide the remainder  condition.
Finally all  conditions of the  theorem \ref{s23} are satisfied  and this completes the proof.
\end{proof}
\section{Numerical results}
\begin{table}.
\caption{ Numerical result obtained from
(\ref{e3a}),(\ref{e3b})}\label{T5}
 \[
\begin{array}{|c|c|c|c|c|c|}
\hline
 n     & 2       & 4                       & 6                        & 8             & 10           \\ \hline
 \|x_n-x\| & 4.0e-002 & 3.0e-004 &7.4e-007 &    9.4e-010    & 7.6e-013\\ \hline
 \|y_n-y\| & 1.6e-001 & 2.2e-003 &8.1e-006 &    1.4e-008    & 1.4e-011 \\ \hline
\end{array}
\]
\end{table}
\begin{table}.
\caption{Numerical result obtained from (\ref{e3a}),(\ref{e3c})
}\label{T6}
\[
\begin{array}{|c|c|c|c|c|c|} \hline

 n     & 2       & 4                       & 6                        & 8       & 10                 \\ \hline
 \|x_n-x\| &  3.8e-002 & 2.6e-004 & 6.7e-007   &  8.8e-010 &  6.8e-013  \\ \hline
 \|y_n-y\| & 7.3e-002  &  5.1e-004    & 1.3e-006  & 1.7e-009 & 1.3e-012 \\ \hline
\end{array}
\]
\end{table}
\begin{table}.
\caption{Logarithm of the error by best approximation using Legendre
polynomial) }\label{T7}
\[
\begin{array}{|c|c|c|c|c|c|} \hline
 n     & 2       & 4                       & 6                        & 8       & 10                 \\ \hline
 \|P_n\sin(t)-\sin(t)\| &  5.1e-002 & 3.3e-004 & 8.1e-007   &  1.0e-009 &  8.0e-013  \\ \hline
 \|P_n\cos(t)-\cos(t)\| & 7.8e-002  &  5.4e-004    & 1.4e-006  & 1.7e-009 & 1.3e-012 \\ \hline
\end{array}
\]
\end{table}
\begin{figure}
\includegraphics [width=2.4in,height=2.5in,]{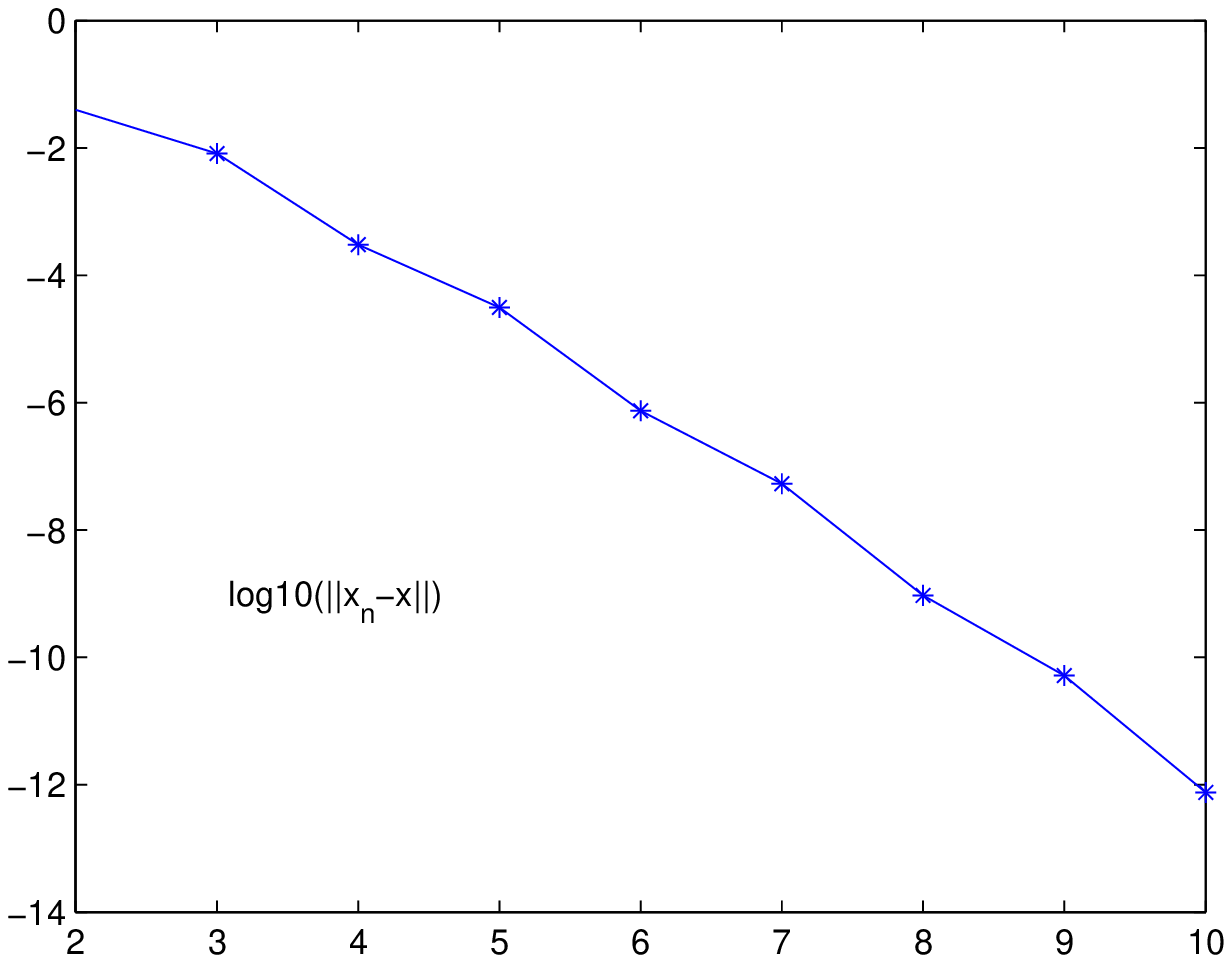}
\includegraphics [width=2.4in,height=2.5in,]{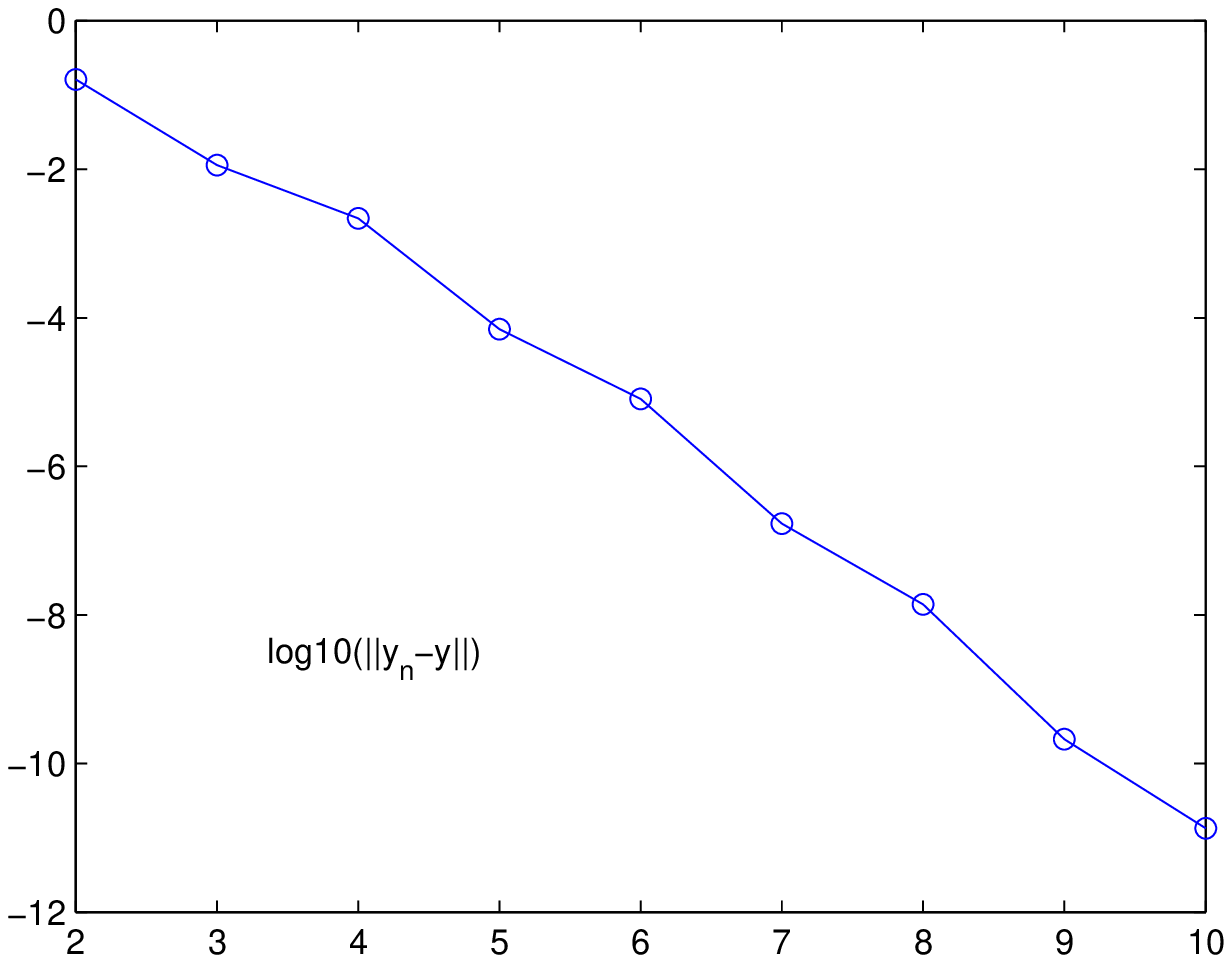}
 \caption{\small{plots of  convergence order by direct Galerkin method}}
\end{figure}
\begin{figure}
\includegraphics [width=2.4in,height=3.1in,]{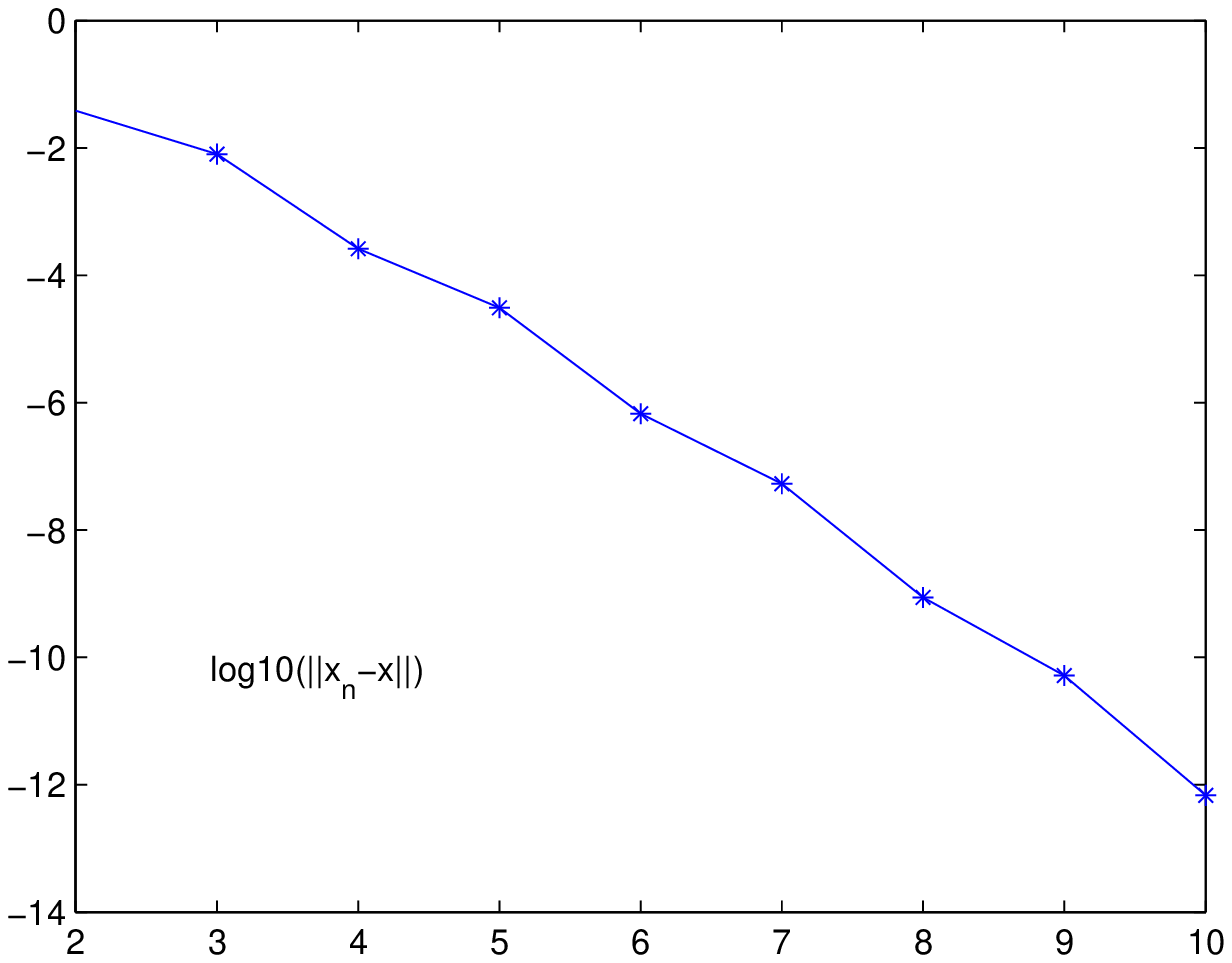}
\includegraphics [width=2.4in,height=3.1in,]{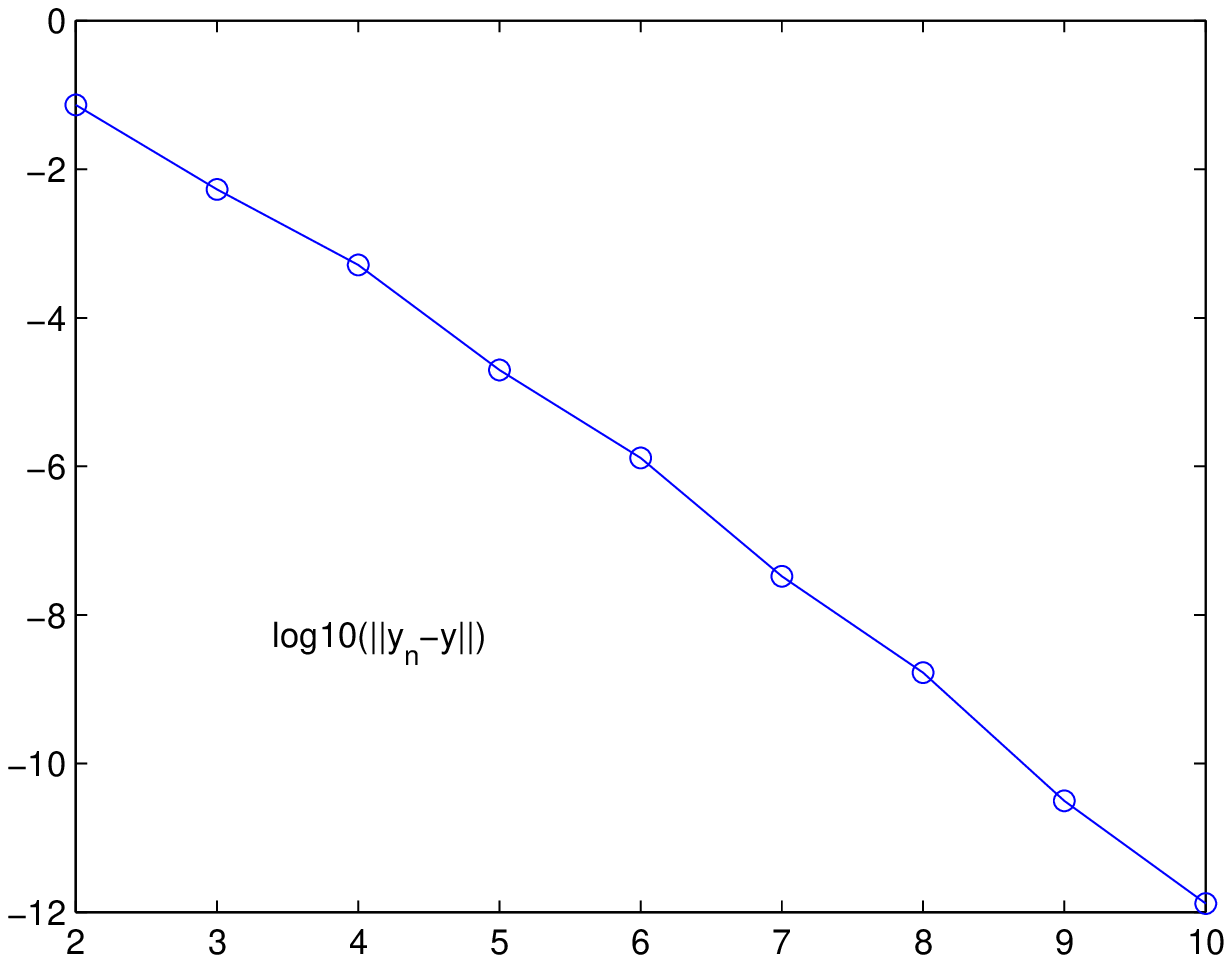}
 \caption{\small{plots of  convergence order by indirect Galerkin method.  We note that the slope of this figures is same as the slope of figure 3 as claimed}}
\end{figure}
\begin{figure}
\includegraphics [width=2.4in,height=3.1in,]{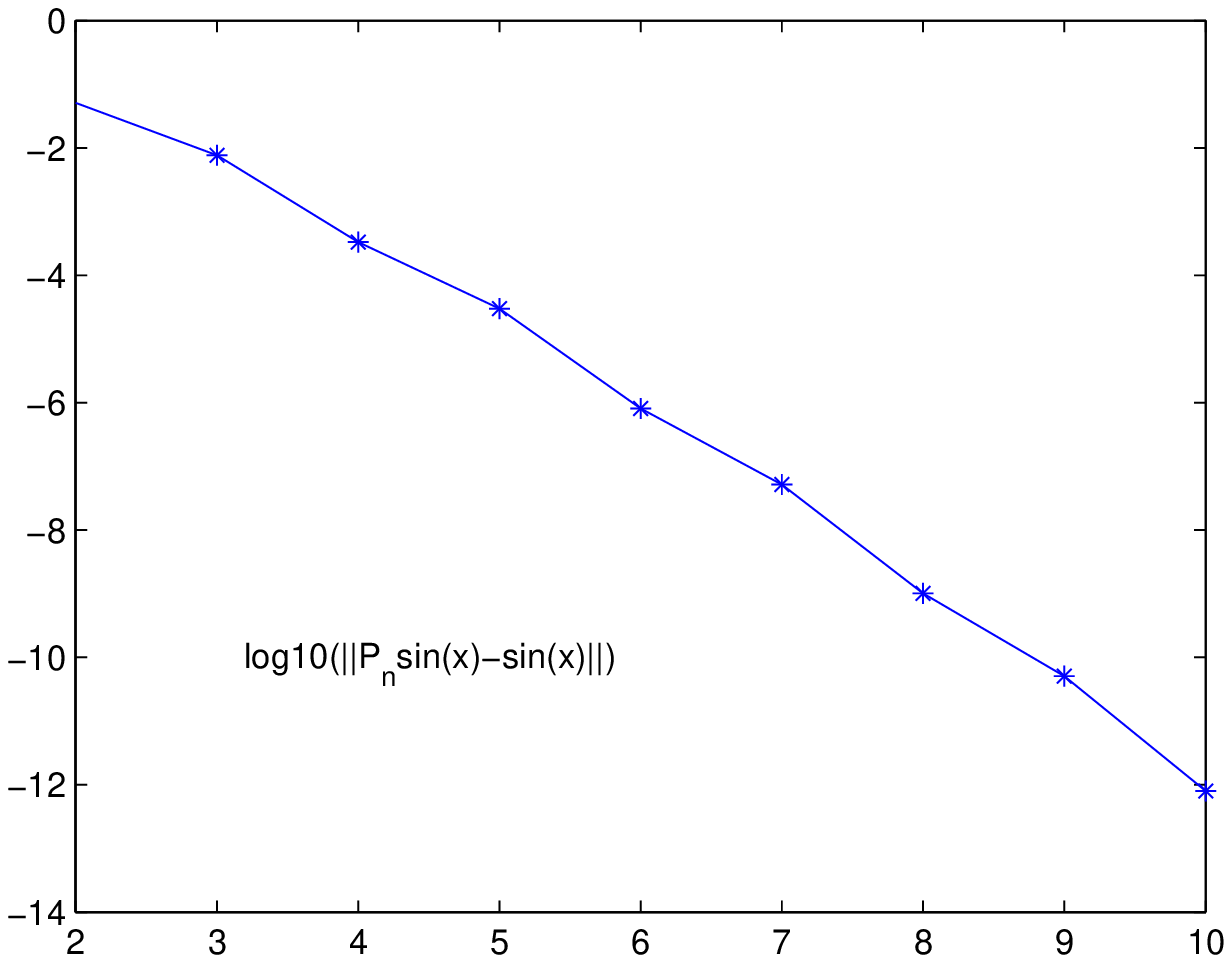}
\includegraphics [width=2.4in,height=3.1in,]{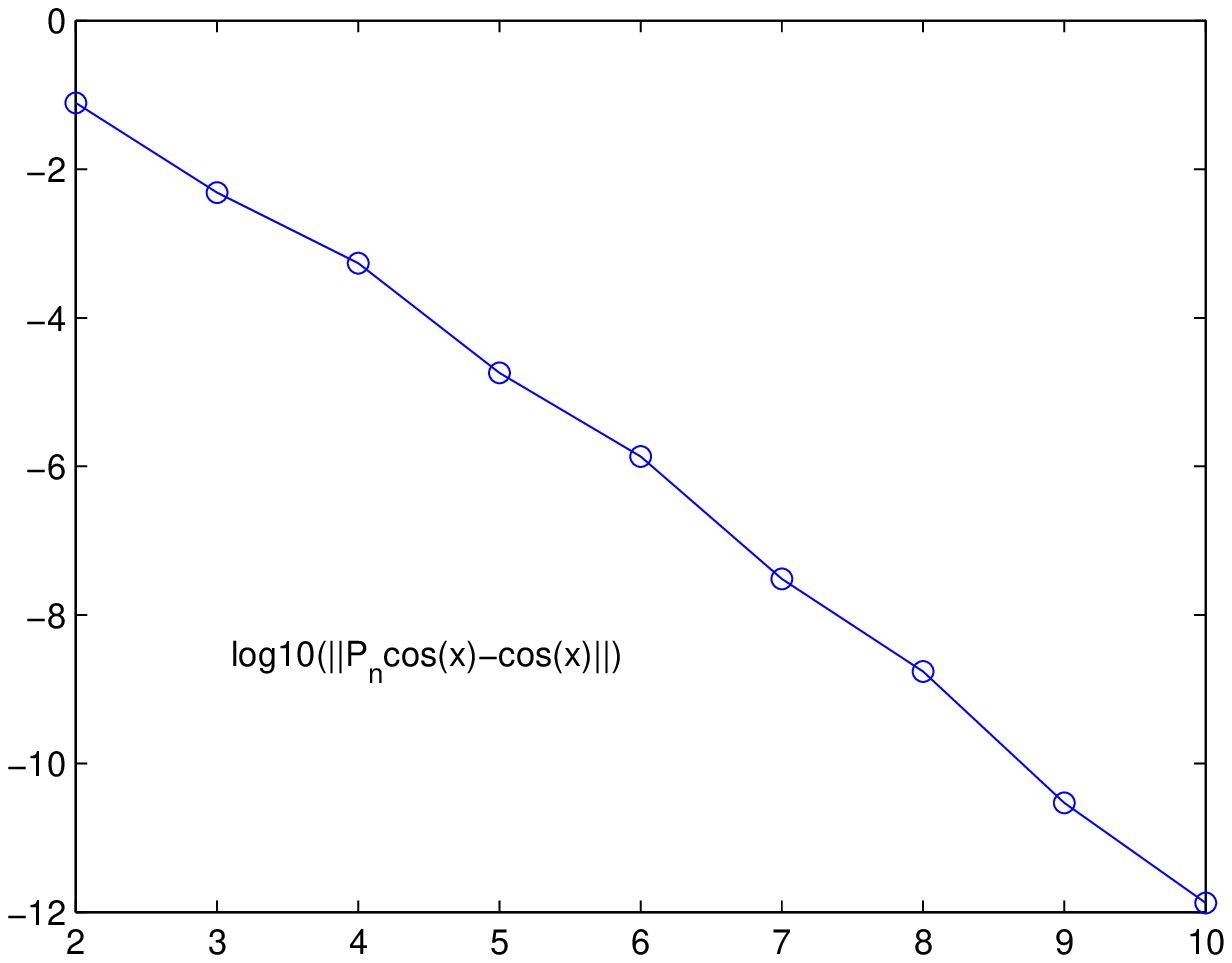}
 \caption{\small{plots of  convergence order of best the approximation}}
\end{figure}
In this section we illustrate the methods by numerical examples. \\
{\bf Example 1.}\\
We consider (\ref{e6}) with $k^{11}=s+t$, $k^{12}=s^2+t^2$,
$k^{21}=s-t^2$, $k^{31}=s+t+1$, $f_1=-t-2sin(t)t^2+2sin(t)$ and
$f_2=t^2-2sin(t)+cos(t)t-cos(t)t^2+1-cos(t)-2sin(t)t$ where the exact
solutions are $x=sin(t)$ and $y=cos(t)$. Numerical  results  obtained by two methods
are in table \ref{T5} and \ref{T6} which show  rapid
convergence. In these tables $\|f\|=\max_{T\in[0, T]}|f(t)|$.
In the following figures,  we  also illustrate $\log_{10}(\|x_n-x\|)$
and $\log_{10}(\|y_n-y\|)$ . For the second  system, we  expect
that this figure to be similar to the figure of
$\log_{10}(\|P_nx-x\|)$ and   $\log_{10}(\|P_ny-y\|)$. This
numerical experiment shows that  the second system is more accurate than the
first. But we need more computations in comparison with the first. \\
\begin{remark}
In equations (\ref{e3a}),(\ref{e3b}) and (\ref{e3c}), we find two
type  of integrals that should be computed, single and double integrals. The integral of type $\int_0^Tf(t)dt$ can be
computed by Gaussian Integration Methods (see \cite{11})
of the form
$$\int_{-1}^1f(t)dt\simeq\sum_{i=1}^nw_if(x_i)$$ and by changing of
variable we have
\begin{equation}\label{gass}
\int_0^Tf(t)dt\simeq\sum_{i=1}^nw_if(Tx_i/2+T/2)
\end{equation}
where $x_1, ..., x_n$ are the roots of $V_{T_(n+1)},$ and  $w_1, ..., w_n$
are computed by  the system expressed in (3.6.13) of \cite{11}. An approximate values of are computed as follows $\int_0^T\int_0^tf(t,s)dsdt$ in this  calculation is as it
follows
\begin{equation}\label{gass}
\int_0^T\int_0^tf(t,s)dsdt=\sum_{i=1}^n\sum_{j=1}^nw_iw_jf(\frac{Tx_i}{2}+\frac{T}{2},\frac{(\frac{Tx_i}{2}+\frac{T}{2})x_i}{2}+\frac{\frac{Tx_i}{2}+\frac{T}{2}}{2})
\end{equation}
Also we notice that the term $n$ is not fixed for different $P_m$,
indeed in our calculation $n=m$ that means for improving accuracy
of the solution we improve the accurate  approximate solutions of
integrals.
\end{remark}

\end{document}